\documentclass[11pt]{amsart}
\usepackage{amssymb}
\usepackage{verbatim}
\usepackage{hyperref}
\usepackage{color}
\usepackage{bbm}

\begin{document}

\newtheorem{theorem}{Theorem}    
\newtheorem{proposition}[theorem]{Proposition}
\newtheorem{conjecture}[theorem]{Conjecture}
\def\theconjecture{\unskip}
\newtheorem{corollary}[theorem]{Corollary}
\newtheorem{lemma}[theorem]{Lemma}
\newtheorem{sublemma}[theorem]{Sublemma}
\newtheorem{observation}[theorem]{Observation}
\theoremstyle{definition}
\newtheorem{definition}{Definition}
\newtheorem{notation}[definition]{Notation}
\newtheorem{remark}[definition]{Remark}
\newtheorem{question}[definition]{Question}
\newtheorem{questions}[definition]{Questions}
\newtheorem{example}[definition]{Example}
\newtheorem{problem}[definition]{Problem}
\newtheorem{exercise}[definition]{Exercise}

\numberwithin{theorem}{section}
\numberwithin{definition}{section}
\numberwithin{equation}{section}

\def\earrow{{\mathbf e}}
\def\rarrow{{\mathbf r}}
\def\uarrow{{\mathbf u}}
\def\tpar{T_{\rm par}}
\def\apar{A_{\rm par}}

\def\reals{{\mathbb R}}
\def\torus{{\mathbb T}}
\def\heis{{\mathbb H}}
\def\integers{{\mathbb Z}}
\def\naturals{{\mathbb N}}
\def\complex{{\mathbb C}\/}
\def\distance{\operatorname{distance}\,}
\def\support{\operatorname{support}\,}
\def\dist{\operatorname{dist}\,}
\def\Span{\operatorname{span}\,}
\def\degree{\operatorname{degree}\,}
\def\kernel{\operatorname{kernel}\,}
\def\dim{\operatorname{dim}\,}
\def\codim{\operatorname{codim}}
\def\trace{\operatorname{trace\,}}
\def\Span{\operatorname{span}\,}
\def\ZZ{ {\mathbb Z} }
\def\p{\partial}
\def\rp{{ ^{-1} }}
\def\Re{\operatorname{Re\,} }
\def\Im{\operatorname{Im\,} }
\def\ov{\overline}
\def\eps{\varepsilon}
\def\lt{L^2}
\def\diver{\operatorname{div}}
\def\curl{\operatorname{curl}}
\def\etta{\eta}
\newcommand{\norm}[1]{ \|  #1 \|}
\def\Span{\operatorname{span}}
\def\expect{\mathbb E}
\def\paraboloid{{\mathbb P}}
\newcommand{\Pp}[1]{\mathbb P^{#1}}
\def\Sbest{{\mathbf S}}
\def\Pbest{{\mathbf P}}

\newcommand{\Norm}[1]{ \left\|  #1 \right\| }
\newcommand{\set}[1]{ \left\{ #1 \right\} }
\def\one{\mathbf 1}
\newcommand{\modulo}[2]{[#1]_{#2}}

\def\scriptf{{\mathcal F}}
\def\scriptg{{\mathcal G}}
\def\scriptm{{\mathcal M}}
\def\scriptb{{\mathcal B}}
\def\scriptc{{\mathcal C}}
\def\scriptt{{\mathcal T}}
\def\scripti{{\mathcal I}}
\def\scripte{{\mathcal E}}
\def\scriptv{{\mathcal V}}
\def\scriptS{{\mathcal S}}
\def\scripta{{\mathcal A}}
\def\scriptr{{\mathcal R}}
\def\scripto{{\mathcal O}}
\def\scripth{{\mathcal H}}
\def\scriptd{{\mathcal D}}
\def\scriptl{{\mathcal L}}
\def\scriptn{{\mathcal N}}
\def\frakv{{\mathfrak V}}

\renewcommand{\leq}{\leqslant}
\renewcommand{\geq}{\geqslant}
\newcommand{\Cc}{\mathbb{C}}
\newcommand{\R}{\mathbbm{R}} 
\providecommand{\ab}[1]{\vert #1\vert}   
\providecommand{\norma}[1]{\Vert #1 \Vert}

\author{Michael Christ}
\address{
        Michael Christ\\
        Department of Mathematics\\
        University of California \\
        Berkeley, CA 94720-3840, USA}
\email{mchrist@math.berkeley.edu}

\author{Ren\'e Quilodr\'an}
\address{
        Ren\'e Quilodr\'an\\
        Department of Mathematics\\
        University of California \\
        Berkeley, CA 94720-3840, USA}
\email{rquilodr@math.berkeley.edu}
\thanks{The authors were supported in part by NSF grant DMS-0901569.}

\date{June 2, 2010.}

\title[Gaussians Rarely Extremize]
{Gaussians Rarely Extremize Adjoint Fourier Restriction Inequalities For Paraboloids}


\maketitle


\section{Introduction}

Let $\Pp{d-1}$ be the paraboloid in $\reals^d$,
\[
\Pp{d-1}=\{(y',y_d)\in\reals^{d-1}\times\reals: y_d = |y'|^2/2\}.
\]
Equip $\Pp{d-1}$ with the appropriately dilation-invariant measure $\sigma$ on $\reals^d$ defined by 
\[\int_{\reals^d} f(y',y_d)\,d\sigma(y',y_d) = \int_{\reals^{d-1}} f(y',|y'|^2/2)\,dy',\]
where $dy'$ denotes Lebesgue measure on $\reals^{d-1}$.

The adjoint Fourier restriction inequality
states that for a certain range of exponents $p$,
\begin{equation} \label{eq:TS}
\norm{\widehat{f\sigma}}_q\le C\norm{f}_{L^p(\Pp{d-1},\sigma)}
\end{equation}
for some finite constant $C=C(p,d)$,
where 
$q=q(p)$ is specified by
\begin{equation} \label{qdef}
q^{-1} = 
\frac{d-1}{d+1}(1-p^{-1}).
\end{equation}
This inequality is known to be valid for $1\le p\le p_0$
for a certain exponent $p_0>2$ depending on $d$,
and is conjectured to be valid for all $p\in [1,\frac{2d}{d-1})$.

The case $p=2$ is of special interest, since it gives a space-time upper bound
for the solution of a linear Schr\"odinger equation
with arbitrary initial data in the natural class $L^2(\reals^{d-1})$.
While the cases $p\ne 2$ also give such bounds, they are expressed in terms
of less natural norms on initial data.

The more general Strichartz inequalities \eqref{strichartz} are phrased in terms of mixed norms.
For simplicity we restrict our discussion of mixed norm inequalities to the case $p=2$.
For $\reals^d$, adopt coordinates $(x,t)\in \reals^{d-1}\times\reals$.
For $r,q\in[1,\infty)$, for $u:\reals^d_{x,t}\to\complex$,
define $\norm{u}_{L^r_tL^q_x}
= (\int (\int |u(x,t)|^q\,dx)^{r/q}\,dt)^{1/r}$.
The Strichartz inequalities state  \cite{tao} 
that
\begin{equation}  \label{strichartz}
\norm{\widehat{f\sigma}}_{L^r_tL^q_x}
\le C\norm{f}_{L^2(\Pp{d-1},\sigma)}
\end{equation}
for all $r,q,d$ satisfying
$q,r\ge 2$ and
\begin{equation}
\label{strichartzexponents}
\frac2{r}+\frac{d-1}{q} = \frac{d-1}2
\end{equation}
with the endpoint $q=\infty$ excluded for $d=3$. 

By a radial Gaussian we mean a function $f:\Pp{d-1}\to\complex$
of the form $f(y,|y|^2/2) = c\exp(-z|y-y_0|^2+y\cdot v)$ for $y\in\reals^{d-1}$,
where $0\ne c\in\complex$, $y_0\in\reals^{d-1}$, and $v\in\complex^{d-1}$ are arbitrary,
and $z\in\complex$ has positive real part.
Radial Gaussians on $\Pp{d-1}$ are simply restrictions to $\Pp{d-1}$ of  
functions $F(x)=e^{x\cdot w+c}$ where $w=(w',w_d)\in\complex^d$ satisfies $\Re(w_d)<0$.

Radial Gaussians extremize \cite{foschi} 
inequality \eqref{eq:TS} for $p=2$ in the two lowest-dimensional cases, $d=2$ and $d=3$.
More than one proof of these facts is known.
It is natural to ask whether these are isolated facts, or whether Gaussians 
appear as extremizers more generally.
Additional motivation is provided by recent work of
Christ and Shao \cite{christshao1},\cite{christshao2}, who have shown the existence of
extremizers for the corresponding inequalities for the spheres $S^1$ and $S^2$.
Their analysis relies on specific information
about extremizers for the paraboloid, which can be read off from
explicit calculations for Gaussians, but which has not been shown to
follow more directly from the inequality itself.
If Gaussians were known to be extremizers for $\Pp{d-1}$,
it should then be possible to establish the existence of extremizers for $S^{d-1}$.

In this paper, we discuss a related question: Are
radial Gaussians critical points for the nonlinear functionals associated to 
inequalities \eqref{eq:TS} and \eqref{strichartz}?
These functionals are defined as follows.
\begin{equation}
\Phi(f)=\Phi_{p,d}(f) = \frac{\norm{\widehat{f\sigma}}_q^q}{\norm{f}_p^q}, 
\end{equation}
where $q=q(p,d)$ is defined by \eqref{qdef}
and
\begin{equation}
\Psi(f) = \Psi_{q,r,d}(f)
= \frac{\norm{\widehat{f\sigma}}_{L^r_tL^q_x}^r}{\norm{f}_2^r}.
\end{equation}
$\Phi$ is defined for all $0\ne f\in L^p(\Pp{d-1},\sigma)$,
while
$\Psi$ is defined for all $0\ne f\in L^2(\Pp{d-1},\sigma)$.
\eqref{eq:TS}  and \eqref{strichartz} guarantee that $\Phi_{p,d},\Psi_{q,r,d}$ 
are bounded functionals, for the ranges of parameters indicated.

By a critical point of $\Phi$ is of course meant a function $0\ne f\in L^p(\Pp{d-1})$
such that for any
$g\in L^p(\Pp{d-1})$,
\begin{equation} \label{defn:criticalpoint}
\Phi(f+\eps g)
= \Phi(f)
+o(|\eps|)
\text{ as } \eps\to 0.
\end{equation}
Here $\eps\in\complex$.
For  $d\ge 3$,
there is a range of exponents for which $\Phi_{p,d}$ is conjectured to be
bounded but for which this is not known \cite{hotstrichartz},\cite{HZ}.
But $\Phi_{p,d}(f)$ is well-defined and finite for any Schwartz function,
so we may still ask whether \eqref{defn:criticalpoint}
holds whenever $f$ is a radial Gaussian
and $g$ is an arbitrary Schwartz function.  This gives a definition of
critical point which is equivalent whenever the functional is bounded.

It is a simple consequence of symmetries of these functionals that for fixed $p,q,r,d$,
one radial Gaussian is a critical point if and only if all are critical points.
Our main result is as follows.
\begin{theorem}
Let $d\ge 2$, let $1<p<2d/(d-1)$, and set $q=q(p,d)$.
Radial Gaussians are critical points for the $L^p\to L^q$ adjoint
Fourier restriction  inequalities if and only if $p=2$.
Radial Gaussians are critical points for the $L^2\to L^r_t{}L^q_x$ 
Strichartz inequalities for all admissible pairs $(r,q)\in (1,\infty)^2$.
\end{theorem}

For spheres $S^{d-1}$, the situation is different;
constant functions are critical points for the analogues of both functionals.

\section{Euler-Lagrange Equations}
We will
show that extremizers must satisfy a certain Euler-Lagrange equation,
then check by explicit calculation whether radial Gaussians satisfy this equation.
In this section we formulate and justify the Euler-Lagrange equations.

Let $g^\vee$ denote the inverse Fourier transform of $g$.
\begin{proposition} \label{prop:unmixedEL}
Let $d\ge 2$, let $1<p<2d/(d-1)$, and set $q=q(p,d)$.
A complex-valued function $f\in L^p(\Pp{d-1})$ with nonzero $L^p(\Pp{d-1})$ norm
is a critical point of $\Phi_{p,d}$ if and only if there exists $\lambda>0$
such that $f$ satisfies the equation 
\begin{equation} \label{EL}
\Big(|\widehat{f\sigma}|^{q-2}\widehat{f\sigma}\Big)^\vee\Big|_{\Pp{d-1}}
= \lambda |f|^{p-2}f \text{ almost everywhere on }\Pp{d-1}.
\end{equation}


For $S^{d-1}$, the same equation likewise characterizes critical points,
except of course that the restriction on the left-hand side is
to $S^{d-1}$, and $q$ can take on any value in $[q(p,d),\infty)$.
\end{proposition}

$\lambda$ is determined by $\norm{f}_p$ and $\Phi(f)$;
multiply both sides of \eqref{EL} by $\bar f$ and integrate with respect to $\sigma$.

Both exponents $q-1,p-1$ are strictly positive, and $q-2>0$.
Moreover, since $f\in L^p$,
$\widehat{f\sigma}\in L^q$, and therefore
$|\widehat{f\sigma}|^{q-2}\widehat{f\sigma}\in L^{q/(q-1)}(\reals^d)$.
Therefore by the Fourier restriction inequality,
the restriction to $\Pp{d-1}$ of
$\Big(|\widehat{f\sigma}|^{q-2}\widehat{f\sigma}\Big)^\vee$
is a well-defined element of $L^{p/(p-1)}(\Pp{d-1},\sigma)$.
Thus the left-hand side of \eqref{EL} is well-defined for any $f\in L^p(\Pp{d-1},\sigma)$.

\begin{proposition} \label{prop:mixedEL}
Let $(q,r,d)$ satisfy the necessary and sufficient conditions \eqref{strichartzexponents}
for the Strichartz inequalities.
A nonzero complex-valued function $f\in L^2(\Pp{d-1})$ 
is a critical point of $\Psi_{p,d}$ if and only if there exists $\lambda>0$
such that $f$ satisfies the equation 
\begin{equation} \label{ELmixedwithFT}
\Big( \widehat{f\sigma}(x,t) |\widehat{f\sigma}(x,t)|^{q-2} \norm{\widehat{f\sigma}(\cdot,t)}_{L^q_x}^{r-q} \Big)^\vee
= \lambda f \text{ a.e.\ on } \Pp{d-1}.  
\end{equation}
\end{proposition}

The Euler-Lagrange equation for $S^{d-1}$ takes the corresponding form.
It follows at once that constant functions are critical points for $S^{d-1}$, 
because $\widehat{\sigma}|\widehat{\sigma}|^{q-2}$ is a radial function,
the inverse Fourier transform of any radial function is radial,
and the restriction of any radial function to $S^{d-1}$ is constant.

Propositions~\ref{prop:unmixedEL} and \ref{prop:mixedEL}
will follow from the following elementary fact.
\begin{lemma} \label{lemma:noop}
For any exponents $q,r\in (1,\infty)$, 
there exists $\gamma>1$ with the following property.
Let $F,G\in L^r_tL^q_x$ of some measure space(s),
and assume that $\norm{F}_{L^r_tL^q_x}\ne 0$.
Let $z\in\complex$ be a small parameter.
Then
\begin{equation}
\norm{F+zG}_{L^r_tL^q_x}^r
= 
\norm{F}_{L^r_tL^q_x}^r
+ r
\iint
\norm{F_t}^{r-q}
|F(x,t)|^q\Re(zG(x,t)/F(x,t))
\,dx\,dt
+O\big(|z|^\gamma\big)
\end{equation}
as $z\to 0$.
\end{lemma}

Here $F_t(x)=F(x,t)$ and $\norm{F_t}^q = \int |F(x,t)|^q\,dx$. 
The constant implicit in the remainder term
$O\big(|z|^\gamma\big)$ does depend on the norms of $F,G$.
It is a consequence of H\"older's inequality that the double integral is absolutely convergent.

An immediate consequence is: 
\begin{proposition}
Let $T$ be a bounded linear operator from $L^p$ to
$L^r_tL^q_x$ where $p,q,r\in(1,\infty)$. For $0\ne f\in L^p$
define $\Phi(f)=\norm{Tf}_{L^rL^q}^r/\norm{f}_p^r$.
Then any critical point $f$ of $\Phi$ satisfies the equation
\begin{equation} \label{eq:mixedEL}
T^*\Big( {Tf}(x,t)\,\, |Tf(x,t)|^{q-2} \,\, \norm{Tf(\cdot,t)}_{L^q_x}^{r-q} \Big)
= \lambda |f|^{p-2}\,f
\end{equation}
for some $\lambda\in [0,\infty)$.
\end{proposition}
Again, it is a consequence of H\"older's inequality that the indicated function
belongs to  the domain $L^{r'}_tL^{q'}_x$ of the transposed operator $T^*$.

\begin{proof}[Proof of Lemma~\ref{lemma:noop}]
Let $\eps>0$ be a small exponent, to be chosen below.
Assume throughout the discussion that $|z|\le 1$.
Write $F_t(x)=F(x,t)$, $G_t(x)=G(x,t)$,
$\norm{F_t}=\norm{F(\cdot,t)}_{L^q_x}$,
and
$\norm{G_t}=\norm{G(\cdot,t)}_{L^q_x}$.
Define
\begin{gather}
\Omega_t=\{x\in\reals^{d-1}: |zG(x,t)|\le |z|^\eps|F(x,t)| \text{ and } F(x,t)\ne 0\}
\\
\omega=\{t: 
\norm{F_t}\ne 0 \text{ and }
|z|\,\norm{G_t}\le |z|^\eps\norm{F_t}\}.
\end{gather}

Fix $\rho\in(0,q-1)$. 
For $x\in\Omega_t$, 
expand
\begin{equation*}
|1+zG(x,t)/F(x,t)|^q
=1+q\Re(zG(x,t)/F(x,t))+O(|zG(x,t)/F(x,t)|^{1+\rho})
\end{equation*}
to obtain
\begin{align*}
\int_{\Omega_t} |(F+zG)(x,t)|^q\,dx
&= 
\int_{\Omega_t} |F(x,t)|^q\,dx
+q
\int_{\Omega_t}|F(x,t)|^{q}\Re(zG(x,t)/F(x,t))\,dx
\\
&\qquad\qquad +O\big(
|z|^{1+\rho}
\norm{F_t}^{q-1-\rho}
\norm{G_t}^{1+\rho}
\big).
\end{align*}

The contribution of $\reals^{d-1}\setminus\Omega_t$ 
is negligible, because of the following three
bounds: 
\begin{equation} \label{neg1}
\int_{\reals^{d-1}\setminus\Omega_t} 
|F(x,t)|^q\,dx
\le 
\int_{\reals^{d-1}\setminus\Omega_t} 
|z|^{(1-\eps)q} 
|G(x,t)|^q\,dx
=
|z|^{(1-\eps)q} 
\norm{G_t}^q;
\end{equation}
similarly
\begin{equation} \label{neg2}
\int_{\reals^{d-1}\setminus\Omega_t} 
|F(x,t)|^q|\Re(zG(x,t)/F(x,t))|\,dx
\le C
|z|^{q-C\eps} 
\norm{G_t}^q;
\end{equation}
and
\begin{equation} \label{neg3}
\begin{aligned}
\int_{\reals^{d-1}\setminus\Omega_t} 
|F(x,t)+zG(x,t)|^q\,dx
&\le 2^q
\int_{\reals^{d-1}\setminus\Omega_t} \big(|F(x,t)|^q+|z|^q|G(x,t)|^q\big)\,dx
\\
&\le C |z|^{q-C\eps} \norm{G_t}^q.
\end{aligned}
\end{equation}
Define 
\[H(t) = \norm{F_t}+\norm{G_t}.\]
Then $\int_\reals H(t)^r\,dt<\infty$. 
We have shown that if
$\eps>0$ is chosen to be sufficiently small, depending on $q$, then
\begin{multline}
\int_{\reals^{d-1}} |F(x,t)+zG(x,t)|^q\,dx
\\
= 
\norm{F_t}^q
+q\int_{\reals^{d-1}}|F(x,t)|^q\Re(zG(x,t)/F(x,t))\,dx
+O\big(|z|^{1+\sigma}\norm{G_t}^{1+\sigma}H(t)^{q-1-\sigma}\big)
\end{multline}
for all sufficiently small $\sigma>0$. 


Suppose that $t\in\omega$.
For any $|z|\ll 1$,
\begin{equation*}
|z|^{1+\sigma}\norm{G_t}^{1+\sigma}\norm{F_t}^{-q}H(t)^{q-1-\sigma}
\le |z|^{(1+\sigma)\eps}\norm{F_t}^{1+\sigma-q}H(t)^{q-1-\sigma}
=O(|z|^{(1+\sigma)\eps})
\ll 1.
\end{equation*}
Similarly,
by H\"older's inequality,
\begin{equation}\label{minbound}
\norm{F_t}^{-q}
\int_{\reals^{d-1}}|F(x,t)|^q\Re(zG(x,t)/F(x,t))\,dx
=O\big(\min(|z|\norm{G_t}\norm{F_t}^{-1},|z|^\eps)\big)\ll 1.
\end{equation}
Therefore for all sufficiently small $z\in\complex$,
\begin{align*}
\norm{F_t}^{-r}(\int_{\reals^{d-1}}  |F+zG|^q\,dx)^{r/q}
&
= 
\Big(
1
+q
\norm{F_t}^{-q}
\int_{\reals^{d-1}}|F(x,t)|^q\Re(zG(x,t)/F(x,t))\,dx
\\
&\qquad\qquad
+O(|z|^{1+\sigma}\norm{G_t}^{1+\sigma}\norm{F_t}^{-q}H(t)^{q-1-\sigma})
\Big)^{r/q}
\\
&= 1 
+ r
\norm{F_t}^{-q}
\int_{\reals^{d-1}}|F(x,t)|^q\Re(zG(x,t)/F(x,t))\,dx
\\
&\qquad
+O(|z|^{1+\sigma}\norm{G_t}^{1+\sigma}H(t)^{-1-\sigma})
\\
&\qquad+
O\big(
\norm{F_t}^{-q}
\int_{\reals^{d-1}}|F(x,t)|^q|zG(x,t)/F(x,t)|\,dx
\big)^2
\\
&= 1 
+ r
\norm{F_t}^{-q}
\int_{\reals^{d-1}}|F(x,t)|^q\Re(zG(x,t)/F(x,t))\,dx
\\
&\qquad
+O(|z|^{1+\sigma}\norm{G_t}^{1+\sigma}\norm{F_t}^{-1-\sigma}),
\end{align*}
provided that $\sigma<1$,
using \eqref{minbound} to deduce the final line.
Provided that $\sigma$ is chosen to satisfy $\sigma < \min(r-1,1)$,
an application of H\"older's inequality now yields
\begin{multline}\label{havefound}
\int_\omega(\int_{\reals^{d-1}}  |F+zG|^q\,dx)^{r/q}
\\
=  
\int_\omega\norm{F_t}^r
+ r
\int_\omega\norm{F_t}^{r-q}
\int_{\reals^{d-1}}|F(x,t)|^q\Re(zG(x,t)/F(x,t))\,dx
+O\big(
|z|^{1+\sigma}
\big).
\end{multline}

It remains to verify that the contribution of $\reals\setminus\omega$ is negliglible.
If $t\notin\omega$ then $\norm{F_t}\le|z|^{1-\eps}\norm{G_t}$,
so
\begin{equation}
(\int_{\reals^{d-1}}|F(x,t)+zG(x,t)|^q\,dx)^{1/q}
\le C|z|^{1-\eps}\norm{G_t}
\end{equation}
and consequently
\begin{equation}
\label{neg4}
\int_{\reals\setminus\omega}(\int_{\reals^{d-1}}|F(x,t)+zG(x,t)|^q\,dx)^{q/r}\,dt
\le C|z|^{(1-\eps)r}
\int_{\reals\setminus\omega}\norm{G_t}^r\,dt
=O(|z|^{(1-\eps)r});
\end{equation}
in the same way,
\begin{equation}
\label{neg5}
\int_{\reals\setminus\omega} \norm{F_t}^r\,dt
\le 
\int_{\reals\setminus\omega} 
|z|^{(1-\eps)r}
\norm{G_t}^r\,dt
=O( |z|^{(1-\eps)r}).
\end{equation}
Finally
\begin{equation} \label{neg6}
\begin{aligned}
\int_{\reals\setminus\omega}
\norm{F_t}^{r-q}
\int_{\reals^{d-1}}|F(x,t)|^q|\Re(zG(x,t)/F(x,t))|\,dx\,dt
&\le |z|
\int_{\reals\setminus\omega}
\norm{F_t}^{r-1}\norm{G_t}\,dt
\\
&\le
|z|^{r-C\eps}
\int_{\reals\setminus\omega}
\norm{G_t}^r\,dt
\\
&=O( |z|^{r-C\eps}).
\end{aligned}
\end{equation}
In conjunction with \eqref{havefound}, the three bounds
\eqref{neg4},\eqref{neg5},\eqref{neg6} complete the proof
once $\eps$ is chosen to be sufficiently small.
\end{proof}

\section{The case $p=2$ and $q=r$}\label{p=3,q=r}
Functions in $L^p(\Pp{d-1},\sigma)$ may be identified with functions
in $L^p(\reals^{d-1})$ via the correspondence
$f(y,|y|^2/2)=g(y)$ for $y\in\reals^{d-1}$.
We will often make this identification without further comment.
Thus a function $g\in L^p(\reals^{d-1})$ is said to satisfy
the equation \eqref{EL}, if the corresponding function
$f(y,|y|^2/2)=g(y)$ does so. We will sometimes write $g\sigma$,
for $g\in L^p(\reals^{d-1})$,
as shorthand for $f\sigma$, where $f,g$ corresponding in this way.

\begin{lemma} \label{lemma:symmetry}
Fix $p,d$ and let $q=q(p,d)$.
Suppose that $f\in L^2(\reals^{d-1})$ satisfies the Euler-Lagrange equation \eqref{EL}.
Then so does the function
$y'\mapsto \rho f(rAy'+v)e^{iy'\cdot w}$ for any $r>0$,
$\rho\in\complex\setminus\{0\}$,
$A\in O(d-1)$, and $v,w\in\reals^{d-1}$.
\end{lemma}

The proof is left to the reader.
To prove our main result,
it suffices to consider henceforth the radial Gaussian
$f(y)=e^{-|y|^2/2}$, $y\in\reals^{d-1}$, for which
$\widehat{f\sigma}(x,t)=u(x,t)$ takes the form
\begin{equation}
\label{gaussianFT}
\begin{split}
u(x,t)=
\widehat{f\sigma}(x,t)
&=
\int e^{-ix\cdot y}e^{-it|y|^2/2}e^{-|y|^2/2}\,dy
\\
&= 
(2\pi)^{(d-1)/2}
(1+it)^{-(d-1)/2}e^{-|x|^2/2(1+it)}.
\end{split}
\end{equation}
Throughout the discussion we will encounter real powers of
$1\pm it$ and of $q-1-it$. 
These are always interpreted as the 
corresponding powers of
$\log(1\pm it)$ and of $\log(q-1-it)$ respectively,
where the branch of log is chosen so that 
$\log(1)=0$ and
$\log(1+it)$
is analytic in the complement of the ray $\{is: s\in [1,\infty)\}$, 
while $\log(1-it)$ and $\log(q-1-it)$
are both analytic in the complement of the ray $\{-is: s\in [1,\infty)\}$, 
with values $0$ and $\log(q-1)$ respectively when $t=0$. 
Thus
\begin{align*}
|u|^{q-2} u
&= 
(2\pi)^{(q-1)(d-1)/2}
(1+t^2)^{-(d-1)(q-2)/4}
(1+it)^{-(d-1)/2}
e^{-|x|^2
\Big(\frac{1-it}{1+t^2}+\frac{q-2}{1+t^2}\Big)/2
}
\\
&=
(2\pi)^{(q-1)(d-1)/2}
(1+t^2)^{-(d-1)(q-2)/4}
(1+it)^{-(d-1)/2}
e^{-|x|^2 (q-1-it)/2(1+t^2) } .
\end{align*}


We now begin to analyze the inverse Fourier transform
$\iint e^{ix\cdot y}e^{it|y|^2/2} |u(x,t)|^{q-2}u(x,t)\,dx\,dt$
by calculating the integral with respect to $x\in\reals^{d-1}$.
\begin{align*}
\int_{\reals^{d-1}} e^{ix\cdot y}
e^{-\tfrac12|x|^2 \frac{q-1-it}{1+t^2} }dx
=
(2\pi)^{(d-1)/2}
\Big(\frac{q-1-it}{1+t^2}\Big)^{-(d-1)/2}
e^{-\tfrac12|y|^2 \frac {1+t^2} {q-1-it} }.
\end{align*}
Thus
\begin{align*}
\big(|u|^{q-2}u\big)^\vee&(y,|y|^2/2)
= 
(2\pi)^{q(d-1)/2}
\\
&
\int_\reals
e^{it|y|^2/2}
(1+t^2)^{-(d-1)(q-2)/4}
(1+it)^{-(d-1)/2}
\Big(\frac{q-1-it}{1+t^2}\Big)^{-(d-1)/2}
e^{-\tfrac12|y|^2 \frac {1+t^2} {q-1-it} }
\,dt
\end{align*}
which simplifies to
\begin{equation} \label{key}
(2\pi)^{q(d-1)/2}
 \int_\reals
(1+it)^{-(d-1)(q-2)/4}
(1-it)^{-\tfrac14(d-1)(q-2) + \tfrac12(d-1)}
(q-1-it)^{-(d-1)/2}
e^{\tfrac12|y|^2 \big( it - \frac {1+t^2} {q-1-it} \big)}
\,dt.
\end{equation}

Consider first the case $p=2$. Then $q = 2(d+1)/(d-1) = 2+\frac{4}{d-1}$,
so $(d-1)(q-2)/4=1$ and the integral with respect to $t\in\reals$ becomes
\begin{align*}
(2\pi)^{q(d-1)/2}
\int_\reals
(1+it)^{-1}(1-it)^{(d-3)/2}(q-1-it)^{-(d-1)/2}
e^{a\big( it - \frac {1+t^2} {q-1-it} \big)}
\,dt
\end{align*}
where $a = |y|^2/2$.
This may be evaluated by deformation of the contour of integration through
the upper half-plane in $\complex$.
In the upper half-plane, the integrand is meromorphic with a single pole at
$t=i$. Therefore the integral equals
\begin{multline*}
(2\pi i)
(2\pi)^{q(d-1)/2}
i^{-1}2^{(d-3)/2}q^{-(d-1)/2}
e^{a \big( i\cdot i - \frac {1+i^2} {q-1-i\cdot i} \big)}
= 
(2\pi)^{d+2}2^{(d-3)/2}q^{-(d-1)/2}
e^{-a}
\\
=
(2\pi)^{d+2}2^{(d-3)/2}q^{-(d-1)/2}
e^{-|y|^2/2}
=
(2\pi)^{d+2}2^{(d-3)/2}q^{-(d-1)/2}
f(y).
\end{multline*}
Since $p=2$,
$f\equiv |f|^{p-2}f$ for $p=2$ and thus
the Euler-Lagrange equation \eqref{EL} is indeed satisfied. 

\medskip
Now consider the general mixed-norm case.
The Euler-Lagrange equation is modified via
the factor $\norma{\widehat{f\sigma}(\cdot,t)}_{L^q_x}^{r-q}$. By \eqref{gaussianFT},
\[\norma{\widehat{f\sigma}(\cdot,t)}_{L^q_x}^{r-q}=\frac{(2\pi)^{\frac{1}{2}(r-q)(d-1)(1+1/q)}}{q^{(d-1)(r-q)/2q}}(1+t^2)^{-\frac{1}{4q}(d-1)(r-q)(q-2)}.\]
Set
\begin{multline} \label{Jintegral}
J(a)=\int_\R (1+i t)^{-\frac{r}{4q}(d-1)(q-2)}(1-it)^{-\frac{r}{4q}(d-1)(q-2)+\frac{1}{2}(d-1)}
\\ \cdot
(q-1-it)^{-\frac{1}{2}(d-1)}e^{a(it-\frac{1+t^2}{q-1-it})}dt.
\end{multline}
Since $p=2$, the Euler-Lagrange equation \eqref{ELmixedwithFT}
is satisfied if and only if $J(a)$ is a constant multiple of $e^{-a}$.
Using the equation \eqref{strichartzexponents} which relates $q$ to $r$, 
$J(a)$ simplifies to
\[J(a)=\int_\R (1+i t)^{-1}(1-it)^{\frac{1}{2}(d-3)}(q-1-it)^{-\frac{1}{2}(d-1)}e^{a(it-\frac{1+t^2}{q-1-it})}dt,\]
which was shown above to be
a constant multiple of $e^{-a}$. 

\section{The case $p\ne 2$}

We will use the following simple lemma.
\begin{lemma}
\label{integration}
Let $H(t):\Cc\to\Cc$ be holomorphic on the upper half plane $\{\Im(t)> 0\}$ and continuous in its closure, and suppose that $\ab{(1+it)^\gamma H(t)}=O(\ab{t}^{-1-\delta})$ as $\ab{t}\to\infty$, for some $\delta>0$. Then for $\gamma>-1$,
\[\int_\R (1+it)^\gamma H(t)dt=-2\sin(\gamma\pi)\int_0^\infty y^\gamma H(i+iy)dy,\]
and for $\gamma=-1$
\[\int_\R (1+it)^\gamma H(t)dt=2\pi H(i).\]
\end{lemma}

This is obtained via contour integration in the region $\{\Im(t)\geq 0\}\setminus\{iy:y\in[1,\infty)\}$.
As a consequence of Lemma~\ref{integration} we have the following: 
Suppose that $H$ is real-valued, nonnegative when restricted to the imaginary axis, and satisfies $H(i)> 0$. 
If $\gamma\geq -1$, then $\int_\R (1+it)^\gamma H(t)dt=0$ if and only if $\gamma\geq 2$ is an integer.

Define $I:[0,\infty)\to \mathbb C$ by
\begin{equation} \label{Iintegral}
I(a)=\int_\R 
(1+i t)^{-\frac{1}{4}(d-1)(q-2)}
(1-it)^{-\frac{1}{4}(d-1)(q-2)+\frac{1}{2}(d-1)}
(q-1-it)^{-\frac{1}{2}(d-1)}e^{a(it-\frac{1+t^2}{q-1-it})}dt
\end{equation}
where $d\geq 2$, and $q> \frac{2d}{d-1}$ is defined by \eqref{qdef}.
The integrand is 
\[O(t^{-\frac{(d-1)(q-2)}{2}}e^{-a(q-1)\frac{1+t^2}{(q-1)^2+t^2}})\]
and since $q>2$ and $\frac{(d-1)(q-2)}{2}>1$, it belongs to $L^1(\R)$ for all $a\geq 0$. We note that $I(\frac{1}{2}\ab{y})$ equals the expression in \eqref{key} up to constant.

Our goal is to demonstrate:
\begin{lemma}
As a function of $a\in[0,\infty)$,
the function $I$ is a constant multiple of $e^{-(p-1)a}$ only if $p=2$.
\end{lemma}

\begin{proof}

{\bf Case 1} : $p<2$.
Consider
\[e^{a}I(a)
=\int_\R (1+i t)^{-\frac{1}{4}(d-1)(q-2)}
(1-it)^{-\frac{1}{4}(d-1)(q-2)+\frac{1}{2}(d-1)}
(q-1-it)^{-\frac{1}{2}(d-1)}e^{a\frac{(q-2)(1+it)}{q-1-it}}dt.\]	
Expanding the exponential in power series and interchanging integral and sum gives
\[e^a I(a)=\sum\limits_{k=0}^{\infty}\frac{a^k}{k!}(q-2)^k I_k,\]
where 
\[I_k=\int_\R (1+i t)^{k-\frac{1}{4}(d-1)(q-2)}H_k(t)dt,\]
with
\[H_k(t)=(1-it)^{-\frac{1}{4}(d-1)(q-2)+\frac{1}{2}(d-1)}(q-1-it)^{-k-\frac{1}{2}(d-1)}.\]
$H_k$ satisfies the hypothesis of Lemma \ref{integration}, and $H_k(iy)>0$ for all $y\geq 0$.

Now $e^aI(a)$ is a constant multiple of $e^{-(p-2)a}$ if and only if there exists $c\in\mathbb C$ such that for all $k\geq 0$,
\begin{equation}
\label{eqn}
I_k=c\Bigl(\frac{2-p}{q-2}\Bigr)^k.
\end{equation}
Let $k_0=\lceil (d-1)(q-2)/4\rceil$, the smallest integer $\geq (d-1)(q-2)/4$, and consider any $k\geq k_0$. By Lemma \ref{integration},
\[I_k=-2\sin(\alpha_k\pi)\int_0^\infty y^{k-\frac{1}{4}(d-1)(q-2)}H_k(i+iy)dy\]
where $\alpha_k=k-(d-1)(q-2)/4$.

Suppose first that $p$ is such that $(d-1)(q-2)/4$ is not an integer, so $I_k\neq 0$. 
Now $\sin(\alpha_{k+1}\pi)=-\sin(\alpha_k\pi)$ and thus $I_k$ is alternating while $c(2-p)^k(q-2)^{-k}$ is not. 
If $p$ is such that $(d-1)(q-2)/4$ is an integer (necessarily $\geq 2$ as $p\neq 2$) and \eqref{eqn} holds 
we get that $c=0$ since $J_k=0$ for $k\geq k_0$. On the other hand, $k_0-1\geq 1$ and $I_{k_0-1}\neq 0$, for
\begin{align*}
I_{k_0-1}=\pi 2^{-\frac{1}{4}(d-1)(q-2)+\frac{1}{2}(d-1)+1}q^{-k_0-\frac{1}{2}(d-1)+1},
\end{align*}
by Lemma \ref{integration}. 

\noindent
{\bf Case 2:} $p>2$.
It is now convenient to work with
\begin{multline*}
e^{(p-1)a}I(a)
=
\\
\int_\R (1+i t)^{-\frac{1}{4}(d-1)(q-2)}(1-it)^{-\frac{1}{4}(d-1)(q-2)+\frac{1}{2}(d-1)}
(q-1-it)^{-\frac{1}{2}(d-1)}e^{a(p-1+it-\frac{1+t^2}{q-1-it})}dt;
\end{multline*}
we need to show that this expression is not constant, as a function of $a\in[0,\infty)$.
For $2<p\leq 2d/(d-1)$, $\frac{(d-1)(q-2)}{4}=\frac{d-1}{4(p-1)}-\frac{d-3}{4}$ 
lies in $[1/2,1)$. 
Therefore the integrand has an integrable singularity at $t=i$, so we may expand the exponential factor
in the integrand in power series to obtain an analogue of $I_k$:
\[ e^{(p-1)a}I(a) =\sum_{k=0}^\infty \frac{a^k}{k!}I_k'\]
where
\[I_k'=\int_\R (1+i t)^{-\frac{1}{4}(d-1)(q-2)}H_k'(t)dt,\]
with

\[H'_k(t)=(1-it)^{-\frac{1}{4}(d-1)(q-2)+\frac{1}{2}(d-1)}(q-1-it)^{-k-\frac{1}{2}(d-1)}(pq-p-q+(q-p)it)^k.\]

$H_k'$ satisfies the hypothesis of Lemma \ref{integration}, is real when restricted to the imaginary axis and nonnegative at least when $k$ is an even integer.

Lemma~\ref{integration} gives 
\begin{equation} \label{endpne}
I_k'=2
\sin(\tfrac{1}{4}(d-1)(q-2)\pi)
\int_0^\infty y^{-\frac{1}{4}(d-1)(q-2)}H_k(i+iy)dy.
\end{equation}
Since $(d-1)(q-2)/4\in [\tfrac12,1)$, the factor $\sin(\frac{1}{4}(d-1)(q-2)\pi)$ is nonzero.
If $k$ is an even positive integer,
then the integrand is nonnegative,
so the integral in \eqref{endpne} is likewise nonzero. 
\end{proof}

\end{document}